\documentclass[12pt]{amsart} 
\usepackage{latexsym,amssymb, amsthm,amsfonts,amsmath,amstext}
\usepackage[mathscr]{euscript}
\usepackage[abbrev]{amsrefs}

\usepackage[all]{xy}

\usepackage{hyperref}

\usepackage{pb-diagram}
\usepackage{enumitem} 
\usepackage[dvipsnames]{xcolor}
\usepackage{tikz-cd}
\usepackage{graphicx}
\numberwithin{equation}{section}
\usepackage[all]{xy}

\newtheorem{Theorem}{Theorem}[section]
\newtheorem{Proposition}[Theorem]{Proposition}
\newtheorem{Lemma}[Theorem]{Lemma}
\newtheorem{Corollary}[Theorem]{Corollary}
\newtheorem{Definition}[Theorem]{Definition}
\newtheorem{Remark}[Theorem]{Remark}

\newtheorem{Notation}[Theorem]{Notation}

\numberwithin{equation}{section}
\newcommand{\rank}{\operatorname{rank}}
\newcommand{\Hom}{\operatorname{Hom}}

\def\mapright#1{\smash{ \mathop{\to}
    \limits^{#1}}}
    \newcommand{\Ext}{\operatorname{Ext}}

\begin{document}

\title[ moduli spaces of $\alpha$-stable coherent systems]{Some moduli spaces of $\alpha$-stable coherent systems on algebraic surfaces}

\author{L. Costa}
\address{Facultat de Matem\`atiques i Inform\`atica,
Departament de Matem\`atiques i Inform\'atica, Gran Via 
585, 08007 Barcelona, SPAIN } \email{costa@ub.edu}
\thanks{$^*$ Partially supported by PID2020-113674GB-I00.}

\author{I. Mac\'ias Tarr\'io}
\address{Facultat de Matem\`atiques i Inform\`atica,
Departament de Matem\`atiques i Inform\`atica, Gran Via
585, 08007 Barcelona, SPAIN} \email{irene.macias@ub.edu}
\thanks{$^{**}$ Partially supported by PID2020-113674GB-I00.}

\author{L. Roa-Leguizam\'on}
\address{Universidade Estadual de Campinas (UNICAMP) \\ Instituto de Matemática, Estatísitica e Computação Científica (IMECC) \\ Departamento de Matem\'atica \\
Rua S\'ergio Buarque de Holanda, 651\ \ 13083-970 Campinas-SP, Brazil.} \email{leoroale@unicamp.br}
\thanks{$^{***}$ partially supported by  FAPESP post-doctoral grant 2024/02475-0}



\date{\today}

\baselineskip=16pt

\begin{abstract}
Let $X$ be a smooth, irreducible, projective algebraic surface, and let $\alpha \in \mathbb{Q}[m]_{>0}$ be a polynomial. In this paper, we determine topological and geometric properties of the moduli space of $\alpha$-stable coherent systems of type $(n; c_{1}, c_{2}, k)$ with $k < n$ on $X$, for sufficiently large values of $\alpha$. We prove that, for $\alpha$ sufficiently large, the moduli space admits a description as a Grassmann bundle over a moduli space of $H$-stable torsion free sheaves. As a consequence, we obtain results on irreducibility, dimension, and the birational structure of these moduli spaces. Our approach relies on establishing a correspondence between $\alpha$-stable coherent systems and extensions of $H$-stable torsion free sheaves by trivial bundles.
\end{abstract}

\maketitle

\section{Introduction}

By a coherent system on an algebraic variety  we mean a pair $(E,V)$ where $E$ is torsion free sheaf together with a linear subspace $V \subset H^0(E)$of prescribed dimension.  These objects were  introduced in the 1990s by Le Potier \cite{LePotier}, Raghavendra and Vishwanath \cite{RV} (under the name of  `` Pairs'')  and King and Newstead \cite{KN} (under the name of ``Brill-Noether pair'') motivated by problems in Brill–Noether theory, Hilbert schemes and gauge theory. The theory was later developed in a systematic way by  Le Potier \cite{LePotier}, Min He \cite{He}, Bradlow, García-Prada, Mu\~{n}oz and Newstead \cite{BGMN}.  

Coherent systems provide a unifying framework for analyzing the relationship between torsion free sheaves and their spaces of sections, refining the classical Brill–Noether theory by incorporating stability conditions. This perspective has yielded profound insights into the geometry of moduli spaces of torsion free sheaves.  The notion of stability for coherent systems depends on a polynomial $\alpha \in \mathbb{Q}[m]$, whose  degree is strictly less than the dimension of the variety \cite[Th\'eoréme 4.2]{He}. Studying how the moduli space changes as $\alpha$ varies (the so-called wall-crossing behavior) has led to a detailed understanding of the birational geometry of moduli spaces and has provided explicit examples of flips and contractions in moduli theory (see for instance \cite{BGMN, CHCH, He, Th}). 

Coherent systems and their moduli spaces on curves have been extensively studied by several authors.  For instance: Bradlow, García-Prada, Mu\~{n}oz, Mercat, and Newstead determined general topological and geometric properties  \cite{BGMN, BGMN1, BGMMN, BGMMN2},  Lange and Newstead \cite{LN}, Teixidor and Newstead \cite{TN} studied  coherent systems on the projective line and on elliptic curves. Bhosle studied coherent systems on a nodal curve \cite{Bh}. Moreover, the study of moduli spaces of coherent systems on curves has been successful to study the existence of special divisors on curves and problems related to syzygy bundles on curves. For a recent account of the theory and the description of some open questions we refer to reader to \cite{B, N, N1}.

The passage from curves to higher-dimensional varieties raises subtle new challenges, both conceptual and technical. When dealing with torsion free sheaves on surfaces, stability conditions and moduli constructions are considerably more intricate: the geometry of surfaces introduces additional features. Indeed, the notion of stability of a torsion free sheaf depends on an ample divisor, and there is a rich wall-crossing behaviour among different moduli spaces, and in general the construction of stable bundles of higher rank is difficult. Despite these difficulties, extending the theory of coherent systems from curves to surfaces is of fundamental importance (see, for instance, \cite{CHCH, MTR, CMR}). It enriches the classical picture of the moduli of sheaves on surfaces and also provides a natural testing ground for conjectures inspired by higher-dimensional Brill--Noether theory.

In this work we contribute to the development of the theory of coherent systems on algebraic surfaces by describing the moduli space of $\alpha$-stable coherent systems for large values of $\alpha$. We focus on coherent systems $(E,V)$ with $k=\dim V < \rank E = n$.

We first show that if an $\alpha$-stable coherent system $(E,V)$ exists, then $\alpha$ is bounded above and satisfies
\[
\alpha \leq \frac{1}{n-k}\big(c_1Hm+\chi(E)-n\chi(\mathcal{O}_X)\big).
\]
(see Proposition \ref{boundalpha}). For values of $\alpha$ close to this bound, we prove that the moduli space is a Grassmann bundle over a suitable moduli space of stable bundles (see Theorems \ref{dimGL} and \ref{dimGL2}).

This description is obtained by establishing relations between coherent systems and $H$-stable torsion free sheaves. One such relation is given by the following result (see Proposition \ref{Fsemistable}):

\noindent\textbf{Proposition 1.}
Let $(E,V)$ be an $\alpha$-semistable coherent system of type $(n;c_1,c_2,k)$ with $k<n$ and
\[
\alpha=\frac{1}{n-k}\big(c_1Hm+\chi(E)-n\chi(\mathcal{O}_X)\big)-\frac{1}{n-k}\epsilon,
\]
where $\epsilon>0$ is sufficiently small. Then $(E,V)$ defines an extension
\[
0\rightarrow \mathcal{O}_X^{\oplus k}\rightarrow E\rightarrow F\rightarrow 0
\]
where $F$ is an $H$-semistable (Mumford--Takemoto) torsion free sheaf.

\medskip

\noindent In the other direction we prove the following result (see Proposition \ref{E-stable}).

\medskip

\noindent\textbf{Proposition 2.}
Let $e=(e_1,\ldots,e_k)\in H^1(F^*)^k$ be the extension class of the exact sequence
\[
e:\; 0\to \mathcal{O}_X^{\oplus k}\to E\to F\to 0,
\]
where the $e_i$ are linearly independent in $H^1(F^*)$. Assume that $F$ is an $H$-stable torsion free sheaf of rank $n-k>0$. Then $e$ defines an $\alpha$-stable coherent system $(E,V)$ for any
\[
\alpha=\frac{1}{n-k}\big(c_1Hm+\chi(E)-n\chi(\mathcal{O}_X)\big)-\frac{1}{n-k}\epsilon,
\]
where $\epsilon>0$ is sufficiently small and satisfies $\epsilon k<n$.

Using the previous result we establish our main theorem (see Theorem \ref{dimGL}).

\noindent\textbf{Theorem 1. }
Let $n, k \in \mathbb{N}$ be  natural numbers with $k<n$, $X$  a smooth projective surface and $H$ an ample  divisor on $X$. Fix Chern classes $c_1$ and $c_2 \gg 0$ and  assume that $$((n-k)K_X+c_1)H \leq 0.$$ Then, the space $G_L^s(n;c_1,c_2,k)$ is isomorphic to a Grassmann bundle over $M_H(n-k;c_1,c_2)$. Moreover, it is irreducible with dimension
$$  \dim G_L^s(n;c_1,c_2,k)=(2n-k)c_2 +(1- n(n-k))\chi(\mathcal{O}_X)-k^2-k(\frac{c_1K}{2}+\frac{c_1^2}{2})-(n-k-1)c_1^2.$$

\vspace{3mm}

In addition, we provide examples where the moduli space of coherent systems is generically smooth (see Proposition \ref{smooth}).

The paper is organized as follows. In Section 2 we review the background on coherent systems on surfaces needed in the sequel. Section 3 contains our main results concerning moduli spaces of $\alpha$-stable coherent systems of type $(n;c_1,c_2,k)$ with $k<n$.

Throughout the paper we work over an algebraically closed field of characteristic $0$, and $X$ denotes a smooth, irreducible, projective algebraic surface.

\section{Coherent Systems}

 The aim of this section is to recall the basic definitions and results on coherent systems on surfaces that will be used throughout the paper. Further details can be found in \cite{He} and \cite{LePotier}.

\begin{Definition}\rm  Let $X$ be a smooth irreducible projective surface. 
    \begin{enumerate}
    \item[(1)] A \it{coherent system of dimension $d$ on $X$} is a pair $(\mathcal{E},\mathcal{V})$, where $\mathcal{E}$ is a coherent sheaf
 of dimension $d$ on $X$,  and
$\mathcal{V} \subseteq H^{0}(X,\mathcal{E})$ is a vector subspace.
\item[(2)] A \it{morphism of coherent systems} $(\mathcal{E}_1,\mathcal{V}_1) \to (\mathcal{E}_2,\mathcal{V}_2)$ is a morphism of coherent sheaves $\phi: \mathcal{E}_1 \to \mathcal{E}_2$ such that $H^0(\phi)(\mathcal{V}_1) \subset \mathcal{V}_2$.
\item[(3)]  A coherent \it{subsystem} of $(\mathcal{E},\mathcal{V})$ is a pair $(\mathcal{F},\mathcal{W})$ 
where $0 \neq \mathcal{F} \subset \mathcal{E}$ is a proper subsheaf of 
$\mathcal{E}$ and $\mathcal{W} \subseteq \mathcal{V}\cap H^{0}(X,\mathcal{F})$.
\item[(4)] A \it{quotient} coherent system  of $(\mathcal{E},\mathcal{V})$ is a coherent system $(\mathcal{G},\mathcal{Z})$ together with a morphism 
 $\phi:(\mathcal{E},\mathcal{V}) \to (\mathcal{G},\mathcal{Z})$ of coherent systems such that the morphism $\phi: \mathcal{E} \to \mathcal{G}$ is surjective and $H^0(\phi)(\mathcal{V})=\mathcal{Z}$.
\end{enumerate}
\end{Definition}

In this paper, we restrict our attention to coherent systems of dimension $2$, that is, to pairs $(E,V)$ where $E$ is a torsion free sheaf. For simplicity, we will denote such systems by $(E,V)$. 

\begin{Remark} \rm
 If $(E,V)$ is a coherent system, then any coherent subsystem $(F,W) \subset (E,V)$ has $F$ torsion free.
\end{Remark}

\begin{Definition} \rm 
A coherent system $(E,V)$ is said to be of type $(n; c_1,c_2, k)$ if $E$ is a  torsion free sheaf of rank $n$ with Chern classes
$c_i \in H^{2i}(X,\mathbb{Z})$ for $i=1,2$ and $V \subset H^0(X,E)$ is a subspace of dimension $k$.
\end{Definition}


\vspace{3mm}

Associated to the coherent systems there is a notion of stability which depends on a parameter $\alpha \in \mathbb{Q}[m]$. To introduce it, it is convenient to fix some notation.

\begin{Notation} \rm
We denote by $\mathbb{Q}[m]$ the space of polynomials in the variable $m$ with coefficients in $\mathbb{Q}$.  
Given two polynomials $p_1, p_2 \in \mathbb{Q}[m]$, we write $p_1 \leq p_2$ if and only if
$p_1(m) - p_2(m) \leq 0$ for $m \gg 0$.  
We say that $\alpha \in \mathbb{Q}[m]_{>0}$ if $\alpha(m) > 0$ for $m \gg 0$. 
    \end{Notation}

    \begin{Definition} \rm
   Let   $\alpha \in \mathbb{Q}[m]_{>0}$ and $H$ an ample divisor  on $X$. Given a coherent system $(E,V)$ of type $(n;c_1,c_2,k)$ on $X$, we define its {\it reduced Hilbert polynomial} by
    \begin{eqnarray*}
P^{\alpha}_{H,(E,V)}(m)=\frac{k}{n}\cdot \alpha +\frac{P_{H,E}(m)}{n},
\end{eqnarray*}
where $P_{H,E}(m)$ denotes the Hilbert polynomial of $E$.
    \end{Definition}

    Notice that, since $X$ is a smooth projective surface, the Hilbert polynomial with respect to $H$ of a  torsion free sheaf $E$ of rank $n$ on $X$ with Chern classes $c_1$ and $c_2$ can be written as follows:  
\[\begin{array}{ll} 
\frac{P_{H,E}(m)}{n} &=\frac{H^2 m^2}{2} + \left[ \frac{c_1 H}{n}- \frac{K_X H}{2} \right]m 
+ \frac{1}{n}\left(\frac{c_1^2-(c_1 K_X)}{2}-c_2 \right)+\chi(\mathcal{O}_X)\\ 
&=\frac{H^2 m^2}{2} + \left[ \frac{c_1 H}{n}- \frac{K_X H}{2} \right]m+ \frac{\chi(E)}{n}\\
\end{array} \]
where $K_X$ denotes the canonical divisor of $X$ and $\chi(F)$ the Euler Characteristic of a sheaf $F$.
 
\begin{Definition} \rm 
Let   $\alpha \in \mathbb{Q}[m]_{>0}$ and $H$ be an ample divisor  on $X$. We say that $(E,V)$ is $\alpha$-stable (resp. $\alpha$-semistable) if for any proper coherent subsystem $0 \neq (F,W)\subset (E,V)$
the following inequality holds
\begin{equation} \label{stability}
    P^{\alpha}_{H,(F,W)}(m) < P^{\alpha}_{H,(E,V)}(m) \,\,\,\,\,  (resp. \leq).
\end{equation}
\end{Definition}

\begin{Definition} \rm 
    Let $\alpha\in\mathbb{Q}[m]_{>0}$. We say that $\alpha$ is a regular value if there exists $\beta_1, \beta_2 \in \mathbb{Q}[m]_{>0}$ with $\beta_1<\alpha < \beta_2$  such that if $(E,V)$ is $\alpha$-stable, then $(E,V)$ is $\gamma$-stable  for any  $\gamma\in(\beta_1,\beta_2)$. If $\alpha$ is not a regular value we say that it is a critical value.
\end{Definition}

\begin{Notation} \rm
    According to \cite{He}; Theorem 4.2, given $(n;c_1,c_2,k)$ there is a finite number of critical values
    \[0 = \alpha_0 < \alpha_1 < \cdots < \alpha_s\]
    with $\alpha_i \in \mathbb{Q}[m]_{>0} $ and  in addition $\deg(\alpha_i) < \dim(X)$. 
\end{Notation}

\section{\texorpdfstring{$\alpha$}{alpha}-stable coherent systems for large \texorpdfstring{$\alpha$}{alpha}}

The first goal of this section is to obtain an upper bound $c$ for the value of $\alpha_s$ for coherent systems of type $(n;c_1,c_2,k)$, $k<n$. Once this will be established, we will describe the moduli space of $\alpha$-stable coherent systems for $\alpha$ close to this bound $c$ and we will relate it with a moduli space of stable bundles (in the sense of Mumford-Takemoto). 

\vspace{3mm}
To this end, first observe that any coherent system fits into an exact sequence of the form
\begin{equation} \label{sucEx} 0 \to N \to V \otimes \mathcal{O}_X \mapright{\varphi} E \to F \to 0\end{equation}
where $N$ is a torsion free sheaf and $F$ is a coherent sheaf, possibly with torsion. This motivates the following definition.

\begin{Definition} \rm
    A coherent system $(E,V)$ is said to be generated if the valuation map $V\otimes \mathcal{O}_X \to E$ is surjective.
\end{Definition}


\begin{Proposition} \label{generated}
  Let $H$ be an ample divisor and let $(E,V)$ be a coherent system of type $(n; c_1,c_2,k)$.
  \begin{itemize}
      \item[(a)] If $(E,V)$ is  generated, then  $ 0\leq c_1Hm+\chi(E)-n\chi(\mathcal{O}_X)$. In particular $c_1H \geq 0$. 
      \item[(b)]   If $k<n$ and $(E,V)$ is $\alpha$-semistable for some $\alpha \in \mathbb{Q}[m]_{>0}$, then $c_1H \geq 0$.
  \end{itemize}
\end{Proposition}
\begin{proof} $(a)$ Let $(E,V)$ be a generated coherent system.  Here, $E$  is a quotient of  the sheaf $V\otimes \mathcal{O}_X:= \mathcal{O}^{\oplus k}_X$ which is $H$-semistable. Thus, 
    \[  \frac{P_{H,\mathcal{O}^{\oplus k}_X}(m)}{k} \leq \frac{P_{H,E}(m)}{n} \]
        which is equivalent to $ 0  \leq c_1Hm+\chi(E)-n\chi(\mathcal{O}_X)$. In particular, the coefficient $c_1H$ must be positive, as claimed.
        
        \noindent $(b)$ Let $E'$ be a torsion free sheaf generated by $V$. In particular, $n_{E'}=\rank E' \leq k <n$. Thus, $(E',V)$ is a subsystem of  $(E,V)$. Say that $(E',V)$ is of type $(n_{E'}, c'_1,c'_2, v)$ with $0<n_{E'}< n$. Let $\alpha=am+b \in \mathbb{Q}[m]_{>0}$. Since $(E,V)$ is $\alpha$-semistable we have $p_{H,(E',V)}^{\alpha}(m) \leq p_{H,(E,V)}^{\alpha}(m)$ which implies $(c'_1H+av)/n_{E'} \leq (c_1H+av)/n$. Since $n_{E'}<n$, this implies $c'_1H/n_{E'}\leq c_1H/n$.  Since $(E',V)$ is generated,  statement $(a)$ implies that $0 \leq c'_1H/n_{E'} \leq c_1H/n$ as  desired. 
    \end{proof}
        


%

Next result will give us the desired bound $c$ for $\alpha_s$ for coherent systems of type $(n;c_1,c_2,k)$, $k<n$. 

\begin{Proposition} \label{boundalpha}
Let $H$ be an ample divisor and let $(E,V)$ be a coherent system of type $(n;c_1,c_2,k)$ with $k<n$.  If $(E,V)$ is $\alpha-$semistable for some $\alpha>0$, then 
 \[\alpha  \leq \frac{1}{n-k}(c_1Hm+\chi(E)-n\chi(\mathcal{O}_X)).\]
\end{Proposition}
\begin{proof} 
    Let $E'$ be a torsion free sheaf generated by $V$.  Then $(E',V) \subset (E,V)$ is of type $(n'; c_1(E'),c_2(E'), k)$ with $n' \leq k$. By Proposition \ref{generated} (a), since $(E',V)$ is generated  we have
     \[\alpha+\chi(\mathcal{O}_X) \leq \alpha \frac{k}{n'}+ \frac{c_1(E')H}{n'}m+\frac{\chi(E')}{n'}.\]
     On the other hand, since  $(E,V)$ is $\alpha-$semistable, we have $P^\alpha_{H,(E',V)}(m) \leq P^\alpha_{H,(E,V)}(m)$. Hence, putting all together we get
\[        \alpha+\chi(\mathcal{O}_X) \leq \alpha \frac{k}{n'}+ \frac{c_1(E')H}{n'}m+\frac{\chi(E')}{n'} \leq \alpha \frac{k}{n}+ \frac{c_1H}{n}m+\frac{\chi(E)}{n}\]
which implies  
\[\alpha  \leq \frac{1}{n-k}(c_1Hm+\chi(E)-n\chi(\mathcal{O}_X)).\]
\end{proof}

\begin{Remark} \label{cota c1} \rm 
Let $(E,V)$ be a coherent system of type $(n;c_1,c_2,k)$ with $k<n$. Since we are interested in $\alpha$-semistable coherent systems  with respect to polynomials of degree less or equal than one, and not only with respect to constants, according to Proposition \ref{generated} (b), from now on we will assume $c_1H>0$.  In particular, $ 0 < c_1Hm+\chi(E)-n\chi(\mathcal{O}_X)$.
\end{Remark}



Our next goal is to prove that for values of $\alpha$ big enough, an $\alpha$-stable coherent system $(E,V)$ defines an exact sequence as (\ref{sucEx}) with $N=0$ and $F$ an $H$-stable torsion free sheaf. To this end, let us start with the following technical lemma. 

\begin{Lemma}\label{epsilon}
Let $n,k,p,p' \in \mathbb{N}$ with $k\leq n$, $p,p'>0$ and $d,q,q' \in \mathbb{Q}[m]$.
  Let  \[\epsilon(p,q)=\left(\frac{n}{k}\right)\left(\frac{dp+(n-k)q}{n-k+p}\right).\]
Then,  
\begin{itemize}
    \item[(i)]   If $d-q>0$, then  $\epsilon(p,q)\leq \epsilon(p',q)$ if and only if $p\leq p'$;
    \item [(ii)] $\epsilon(p,q)\leq \epsilon(p,q')$ if and only if $q\leq q'$.
\end{itemize} 
\end{Lemma}

    

\begin{proof}
Since $d-q>0$, $(i)$  follows from the fact that $ \epsilon(p,q) \leq \epsilon(p',q)$ if and only if $(p - p')(d - q) \leq 0$. The equivalence $(ii)$ is a direct computation. 
\end{proof}



\begin{Proposition}
     Let $(E,V)$ be a coherent system of type $(n;c_1,c_2,k)$ with $k<n$ which is $\alpha$-semistable for some $\alpha>0$. Then,
there exists $\alpha_I  \in \mathbb{Q}[m]_{>0}$ with  $$ 0 < \alpha_I \leq \frac{1}{n-k}(c_1Hm+\chi(E)-n\chi(\mathcal{O}_X))$$ such that, if $\alpha>\alpha_I$, the map $\varphi: V \otimes \mathcal{O}_X \to E$ is injective.
\end{Proposition}
\begin{proof} Keeping the notation introduced in Lemma \ref{epsilon}, we define
\[\alpha_I:= \max_{0 \leq l <k} \{ \frac{1}{n-k}(d-\epsilon(1,(k-l)\chi(\mathcal{O}_X))\}\]
and assume that $\alpha > \alpha_I$. Let us denote by $N$ the kernel of $\varphi: V \otimes \mathcal{O}_X \to E$  and set $k_0:=rk(N)$. 
By the exact sequence $$0\rightarrow N\rightarrow \mathcal{O}_X^{\oplus k}\rightarrow Im(\varphi)\rightarrow0,$$ we get  $rk(Im(\varphi))=k-k_0$. 
Hence, $(Im(\varphi),V)$ is a generated coherent subsystem of $(E,V)$ of type $(k-k_0;c_1(Im(\varphi)),c_2(Im(\varphi)),k)$. Since $(E,V)$ is $\alpha-$semistable, we have
        \[P^\alpha_{H,(Im(\varphi),V)}(m) \leq P^\alpha_{H,(E,V)}(m), \]
        \noindent which is equivalent to 
\[\frac{k}{k-k_0}\alpha  +  \frac{c_1(Im(\varphi)) H}{k-k_0} m+ \frac{\chi(Im(\varphi))}{k-k_0} \leq 
        \frac{k}{n}\alpha  + 
       \frac{c_1 H}{n} m+ \frac{\chi(E)}{n} \]
       \noindent and to the fact that 
        \[ \alpha \leq \frac{[ c_1 H m+\chi(E)](k-k_0)-n[c_1(Im(\varphi)) H m+\chi(Im(\varphi))]}{k(n-k+k_0)}.\]  

Consider $d= c_1(E) H m+\chi(E)$ and $-d_0=c_1(Im(\varphi)) H m+\chi(Im(\varphi))$. According to this notation,  
\begin{equation}
\label{cota_alpha}
    \alpha\leq \frac{(k-k_0)d-n(-d_0)}{k(n-k+k_0)},
\end{equation}
which is equivalent to
    \[         \alpha \leq \frac{d-\epsilon(k_0,-d_0)}{n-k}.    \]

Since $(Im(\varphi),V)$ is generated,  by Proposition \ref{generated}, $-d_0 \geq (k-k_0)\chi(\mathcal{O}_X)$ and hence, by Lemma \ref{epsilon} (ii), it follows that
\[         \alpha \leq \frac{d-\epsilon(k_0,-d_0)}{n-k} \leq \frac{d-\epsilon(k_0,(k-k_0)\chi(\mathcal{O}_X))}{n-k}.    \]
\noindent {\bf Claim:}     $d > (k-k_0)\chi(\mathcal{O}_X)$.

\noindent Assume that  $d \leq  (k-k_0)\chi(\mathcal{O}_X)$.  By Remark \ref{cota c1}, $ 0 < c_1Hm+\chi(E)-n\chi(\mathcal{O}_X)$. Hence, we would have
\[    n \chi(\mathcal{O}_X) < d \leq  (k-k_0)\chi(\mathcal{O}_X)\]
which in particular implies that $\chi(\mathcal{O}_X) <0$ and hence $c_1Hm+\chi(E)=d <0$ which contradicts the assumption $c_1H>0$ (see Remark \ref{cota c1}). 

It follows from the claim and Lemma  \ref{epsilon} (i) that, for any $p \leq k_0$,
    \begin{equation}
    \label{cota_alpha2}
         \alpha \leq \frac{d-\epsilon(k_0,-d_0)}{n-k} < \frac{d-\epsilon(k_0,(k-k_0)\chi(\mathcal{O}_X))}{n-k} < \frac{d-\epsilon(p,(k-k_0)\chi(\mathcal{O}_X))}{n-k}. 
    \end{equation}
    If $k_0 \geq 1$, taking $p=1$, by (\ref{cota_alpha2}), 
    \begin{equation}
         \alpha < \frac{d-\epsilon(1,(k-k_0)\chi(\mathcal{O}_X))}{n-k} \leq \alpha_I 
    \end{equation}
    which contradicts the fact that $\alpha > \alpha_I$. Thus, $k_0=0$ and  $\varphi$ is injective. 
\end{proof}

To fix some notation, it is convenient to recall the following definition (see \cite[Definition 1.15]{HL}).

\begin{Definition} \rm 
     The saturation of a subsheaf $F \subset E$ is the minimal subsheaf $F'$ contai-\\ning $F$ such that $E/F'$ is pure of dimension $d = dim (E)$ or zero. The saturation of $F$ is the kernel of the surjection
     \[E \to E/F \to  (E/F)/T_{d-1}(E/F)\]
     where $T_i(E)$ is the maximal subsheaf of $E$ of dimension $\leq i$. We will denote the saturation of a sheaf $F$ by $\overline{F}$.
\end{Definition}

Let $f:E \to F$ be a morphism of torsion free sheaves and consider its  canonical factorization 
\[\begin{diagram}
    \node{0} \arrow{e}  \node{Ker(f)} \arrow{} \arrow{e} \node{E} \arrow{s} \arrow{e} \node{Im(f)} \arrow{s} \arrow{e} \arrow{e} \node{0} \\
    \node{0}   \node{F/Im(f)/T_{d-1}(F/Im(f))}  \arrow{w} \node{F} \arrow{w} \node{\overline{Im(f)}} \arrow{w} \arrow{w} \node{0} \arrow{w}
\end{diagram}\]
where $rank \, Im(f)= rank \, \overline{Im(f)}$.

\begin{Proposition}
    Let $(E,V)$ be a coherent system of type $(n;c_1,c_2,k)$ with $k<n$.  Suppose
that the map $\varphi: V \otimes \mathcal{O}_X \to E$ is injective. Then, there exists\[
\alpha_T \in \left( 0, \frac{1}{n-k} \left( c_1 H m + \chi(E) - n \chi(\mathcal{O}_X) \right) \right)
\]
 such that, if $\alpha > \alpha_T $ and $(E,V)$ is $\alpha-$semistable,  then the  cokernel $E/Im (\varphi)$ is torsion free.
\end{Proposition}

\begin{proof}
    Consider the subsystem $(\overline{Im(\varphi)}, V)$.  
    Since $\varphi: V \otimes \mathcal{O}_X \to E$ is injective we have the exact sequence 
    \begin{equation}\label{saturation}
        0 \to V \otimes \mathcal{O}_X \to \overline{Im(\varphi)} \to T \to 0 
    \end{equation}
    where $T=T_{1}(E/Im(\varphi))$ is the torsion part of $E/Im (\varphi)$, i.e it is the maximal subsheaf of $E/Im(\varphi)$ of dimension less than or equal to 1.  

    Since $(E,V)$ is $\alpha-$semistable we have
    \[
        P^\alpha_{H,(\overline{Im(\varphi)}, V)}(m) \leq P^\alpha_{H,(E, V)}(m) \]
    which is equivalent to 
    \[
        \alpha \leq \frac{c_1Hm+\chi(E)-\frac{n}{k}(c_1(\overline{Im(\varphi)})Hm+\chi(\overline{Im(\varphi)})}{n-k} .
\]
Denote the Chern classes of $\overline{Im(\varphi)}$ by $d_i=c_i(\overline{Im(\varphi)})$ for $i=1,2$. Since 
\[ \chi(\overline{Im(\varphi)})= \frac{d_1^2-d_1K_X}{2}-d_2+k\chi(\mathcal{O}_X) \]
the above inequality is equivalent to
 \[        \alpha \leq \frac{1}{n-k} \left ( c_1Hm+\chi(E)-n\chi(\mathcal{O}_X)-\frac{n}{k}(d_1Hm+\frac{d_1^2-d_1K_X}{2}-d_2 )\right ).\]
 If $T$ is supported in dimension $0$, then $d_1=c_1(T)=0$ and  $d_2=c_2(T)\leq0$.  Hence, if
 $\alpha > \frac{1}{n-k}(c_1Hm+\chi(E)-n\chi(\mathcal{O}_X)-\frac{n}{k})$, then 
    \[\frac{1}{n-k}(c_1Hm+\chi(E)-n\chi(\mathcal{O}_X)-\frac{n}{k}) <
        \alpha \leq \frac{1}{n-k}(c_1Hm+\chi(E)-n\chi(\mathcal{O}_X)+\frac{n}{k}d_2), 
\]
which implies that  $-1<d_2 \leq 0$. Therefore, $d_2=0$ and  $E/Im (\varphi)$ is torsion free.

If $T$ is supported in dimension $1$, then  $d_1=c_1(T)$ is an effective divisor and in particular $d_1H \geq 0$. If  $\alpha > \frac{1}{n-k}(c_1Hm+\chi(E)-n\chi(\mathcal{O}_X)-\frac{n}{k})$ then, 
    {\tiny \[\frac{1}{n-k}\big(c_1Hm+\chi(E)-n\chi(\mathcal{O}_X)-\frac{n}{k}\big) <
        \alpha \leq \frac{1}{n-k}\big(c_1Hm+\chi(E)-n\chi(\mathcal{O}_X)-\frac{n}{k}(d_1Hm+\frac{d_1^2-d_1K_X}{2}-d_2 )\big).\] }
        In particular, the polynomial  $$-d_1Hm-\frac{d_1^2-d_1K_X}{2} +1+d_2 >0$$
        which implies that $d_1H \leq 0$. Therefore, since $d_1$ is effective and $H$ ample, $d_1=0$ and  $E/Im (\varphi)$ is torsion free.
\end{proof}

\begin{Proposition}\label{Fsemistable}
    Let $(E,V)$ be an $\alpha$-semistable coherent system of type $(n;c_1,c_2,k)$ with $k<n$ and $\alpha=\frac{1}{n-k}\big(c_1Hm+\chi(E)-n\chi(\mathcal{O}_X)\big)-\frac{1}{n-k}\epsilon$ with $\epsilon>0$ sufficiently small. Then, $(E,V)$ defines an extension $$0\rightarrow \mathcal{O}_X^{\oplus k}\rightarrow E\rightarrow F\rightarrow 0$$
with $F$ an $H$-semistable (Mumford-Takemoto) torsion free sheaf.
\end{Proposition}
\begin{proof}
    By the above results and the fact that $\alpha$ is big enough, we can assume that $(E,V)$ defines the extension 
    \begin{equation}
        0\rightarrow \mathcal{O}_X^{\oplus k}\rightarrow E\rightarrow F\rightarrow 0
        \label{suc1}
    \end{equation}
     with $F$ torsion free. It only remains to see that $F$ is $H$-semistable. We can assume that $k<n-1$. Otherwise, $F$ has rank $1$ and there is nothing to prove. 

     Let $F'\subset F$ be any subbundle with quotient torsion free. Denote the rank and degree of $F'$ by $n'-k$ and $d'$, respectively. Note that $k<n'<n$.

     By the pullback construction, we obtain a subextension 
     \begin{equation}
         0\rightarrow \mathcal{O}_X^{\oplus k}\rightarrow E'\rightarrow F'\rightarrow 0,
         \label{suc2}
     \end{equation}
     where $E'$ is a subbundle of $E$. Notice that $V\cap H^0(E')=V'$, so that $F'$ determines a  coherent subsystem $(E',V)$ of type $(n';c_1',c_2',k)$. In addition, observe that from the exact sequences (\ref{suc1}) and (\ref{suc2}) we have the following equalities:
$$
    \chi(E)= \chi(F)+k\chi(\mathcal{O}_X), \quad  \chi(E')= \chi(F')+k\chi(\mathcal{O}_X), \quad 
    $$
and for $i=1,2$
$$ c_i(E)=c_i(F) \quad \mbox{and} \quad  c_i(E')=c_i(F'). $$

Since $(E',V)$ is a subcoherent system of $(E,V)$ and $(E,V)$ is $\alpha$-semistable, we have $$P^{\alpha}_{(E',V)}(m)\leq P^{\alpha}_{(E,V)}(m),$$
which is equivalent to $$c_1'Hm+\chi(F')\leq \frac{n'}{n}(c_1Hm+\chi(F))+(n'-n)\frac{k}{n}(\alpha+ \chi(\mathcal{O}_X)).$$
In particular,
$$c_1'H\leq \frac{n'}{n}c_1H+(n'-n)\frac{k}{n}\frac{1}{n-k}c_1H,$$
which is equivalent to $$\frac{c_1'H}{n'-k}\leq \frac{c_1H}{n-k}$$ and hence $F$ is $H$-semistable.
\end{proof}

\vspace{3mm}

We have already seen that for $\alpha$ big enough, an $\alpha-$semistable coherent system $(E,V)$ defines an extension $$0\rightarrow \mathcal{O}_X^{\oplus k}\rightarrow E\rightarrow F\rightarrow 0$$
with $F$ an $H$-semistable (Mumford-Takemoto) torsion free sheaf. Our next goal is to reverse the construction. To do so, we start with the following technical result.


\begin{Lemma}\label{rankdim}
    Let $W$ be a coherent sheaf such that $W \subset \mathcal{O}_X^{\oplus k}$. Then, $c_1(W)\cdot H \leq 0$ and $h^0(W) \leq rk(W)$. Moreover, if $c_1(W)\cdot H=0$ and 
    $k':=h^0(W)=rk(W)$, then $W \cong  \mathcal{O}_X^{\oplus k'}$ or $\chi(W)=k'\chi(\mathcal{O}_X)-c_2(W)$ with $c_2(W) \geq 0$ 
\end{Lemma}
\begin{proof}    The first inequality follows from the fact that the vector bundle $ \mathcal{O}_X^{\oplus k}$ is $H$-semistable with first Chern class equal to 0. On the other hand, since $W \subset \mathcal{O}_X^{\oplus k}$, we have $H^0(W) \subset H^0(\mathcal{O}_X^{\oplus k})=\mathbb{C}^k$ and any section of $W$ is nowhere vanishing. Thus, sections of $W$ generate a submodule $\mathcal{O}_X^{k'}$of $W$ of rank $k'$ equal to $h^0(W)$. Therefore, $h^0(W) \leq rk(W)$. 

Assume that $c_1(W)\cdot H=0$ and  $k':=h^0(W)=rk(W)$. Then, we must have the exact sequence 
\[ 0 \longrightarrow \mathcal{O}_X^{k'} \longrightarrow W \longrightarrow T \longrightarrow 0 \]
with $T$ a torsion sheaf. If $\dim Supp(T)=1$, then $c_1(T)=c_1(W)$ is an effective divisor and since $H$ is an ample divisor we have $c_1(T)H>0$ which contradicts the fact that $c_1(W)H=0$. If $\dim Supp(T)=0$, then $c_1(T)=0$ and $c_2(T) \geq 0$. If $c_2(T)=0$, then  $W \cong  \mathcal{O}_X^{\oplus k'}$. Finally, if $c_2(T)=c_2(W)>0$,
\[\chi(W)=\chi(\mathcal{O}_X^{k'})+\chi(T)=k'\chi(\mathcal{O}_X)-c_2(W)\]
with $c_2(W)>0$. 
\end{proof}


\begin{Proposition}
\label{E-stable} 
Let $e=(e_1, \ldots, e_k) \in H^1(F^*)^k$ be the extension class of the exact sequence $$e: 0 \to \mathcal{O}_X^{\oplus k}  \to E \to F \to 0,$$ where the $e_i$ are linearly independent in $H^1(F^*)$. Assume that $F$ is an $H$-stable torsion free sheaf of rank $n-k>0$. Then, $e$ defines an $\alpha$-stable coherent system $(E,V)$ for any $\alpha=\frac{1}{n-k}(c_1Hm+\chi(E)-n\chi(\mathcal{O}_X))-\frac{1}{n-k}\epsilon$ with $\epsilon>0$ sufficiently small (in particular$\epsilon k <n$). 
\end{Proposition}
\begin{proof}
    Let $(E,V)$ be the corresponding coherent system of type $(n;c_1,c_2,k)$ defined by $e=(e_1, \ldots, e_k) \in H^1(F^*)^k$ and let $(E',V')$ be a subobject of $(E,V)$  of type $(n';c_1',c'_2,k')$ with $V'=V\cap H^0(X,E')$. By the pull-back construction,  $E'$ determines an extension 
    \[0 \to W \to E' \to F' \to 0\] with $F' \subset F$ and $W \subset \mathcal{O}_X^{\oplus k}$ which fits into the following diagram
    \begin{equation}\label{diagram}
        \begin{diagram}
        \node{0} \arrow{e}  \node{W} \arrow{e} \arrow{s} \node{E'} \arrow{e} \arrow{s} \node{F'} \arrow{e} \arrow{s} \node{0}\\
        \node{0} \arrow{e}  \node{\mathcal{O}_X^{\oplus k}} \arrow{e}  \node{E} \arrow{e}  \node{F} \arrow{e}  \node{0}
    \end{diagram}.
    \end{equation}
    Taking cohomology in (\ref{diagram}) and keeping in mind that $V'=V\cap H^0(E')$ and $$H^0(W) \hookrightarrow V \cong H^0(\mathcal{O}_X^{\oplus k})$$ we get  $k'=\dim V'=h^0(W)$. On the other hand, since $F$ is $H$-stable,  there exists $\delta >0$ such that
    \begin{equation}\label{Fstable}
        \mu_H(F') \leq \mu_H(F)-\frac{\delta}{n-k}
    \end{equation}
    for all $0\neq F'\subset F$.  Note that the constant $\delta$ can be chosen independently of $F'$.  In fact, consider the set $\{\mu_H(F) | F \subset E \}$. This set has a maximal element $\mu_{\max}$ and the set $\{F | F \subset E,\mu_H(F) = \mu_{\max} \}$  contains a torsion free sheaf $F'$ of largest rank. This element is the largest element in this set with respect to the inclusion relation. Hence, one can set $0<\frac{\delta}{n-k}=\mu_H(F)-\mu_H(F')$. 

    Let us fix $\alpha=\frac{1}{n-k}(c_1Hm+\chi(E)-n\chi(\mathcal{O}_X))-\frac{\epsilon}{n-k}$ where $\epsilon \in \mathbb{Q}_{>0}$. We are going to see that $P^\alpha_{H,(E',V')}(m) < P^\alpha_{H,(E,V)}(m)$. To this end, we will consider different cases separately. 
    
    \noindent \textbf{Case (A): $0 \neq F' \subset F$ and $W = \mathcal{O}_X^{\oplus k'}$.}
    
    From diagram (\ref{diagram}) we have  $c'_1:=c_1(E')=c_1(F')$ and $rk(F')=n'-k'$. Thus, from  (\ref{Fstable}) we get
 \[ {\tiny \begin{array}{rl}
        P^\alpha_{H,(E',V')}(m) - P^\alpha_{H,(E,V)}(m)  =  &\left[\frac{c'_1}{n'}-\frac{c_1}{n}\right]Hm+\left[\frac{\chi(E')}{n'}-\frac{\chi(E)}{n}\right]+\left[\frac{1}{n-k}(c_1Hm+\chi(E)-n\chi(\mathcal{O}_X))-\frac{\epsilon}{n-k}\right]\left[\frac{k'}{n'}-\frac{k}{n}\right] \\ \leq 
        &\left[\frac{c_1}{n-k}\frac{n'-k'}{n'}-\frac{c_1}{n}\right]Hm+\left[\frac{\chi(E')}{n'}-\frac{\chi(E)}{n}\right]-\frac{\delta}{n-k}\frac{n'-k'}{n'}m+ \\ &
         \left[\frac{1}{n-k}(c_1Hm+\chi(E)-n\chi(\mathcal{O}_X))-\frac{\epsilon}{n-k}\right]\left[\frac{k'}{n'}-\frac{k}{n}\right]\\ =
        &\left[-\frac{\delta}{n-k}\frac{n'-k'}{n'}\right]m+\left[\frac{\chi(E')}{n'}-\frac{\chi(E)}{n}\right]+
        \left[\frac{1}{n-k}(\chi(E)-n\chi(\mathcal{O}_X))-\frac{\epsilon}{n-k}\right]\left[\frac{k'}{n'}-\frac{k}{n}\right].
    \end{array}}\]

    Since $\delta >0$ and $n'-k'>0$, we conclude that $P^\alpha_{H,(E',V')}(m) - P^\alpha_{H,(E,V)}(m)<0$.

    \noindent \textbf{Case (B): $0 \neq F' \subset F$ and $W \neq \mathcal{O}_X^{\oplus k'}$.} Let  $l:=rk(W)$. By Lemma \ref{rankdim}, $k' \leq l$ and $c_1(W)H \leq 0$. Hence, 
    \[\mu_H(E')=\frac{(c_1(W)+c_1(F'))H}{n'}\leq \frac{c_1(F')\cdot H}{n'}\] and since $F$ is stable from (\ref{Fstable}) we get
    \[\mu_H(E')\leq  \left(\mu_H(F)-\frac{\delta}{n-k}\right)\frac{n'-l}{n'}.\]
    Therefore, we have
    \[ {\tiny \begin{array}{rl}
        P^\alpha_{H,(E',V')}(m) - P^\alpha_{H,(E,V)}(m)  =  &\left[\frac{c'_1}{n'}-\frac{c_1}{n}\right]Hm+\left[\frac{\chi(E')}{n'}-\frac{\chi(E)}{n}\right]+\left[\frac{1}{n-k}(c_1Hm+\chi(E)-n\chi(\mathcal{O}_X))-\frac{\epsilon}{n-k}\right]\left[\frac{k'}{n'}-\frac{k}{n}\right] \\ \leq 
        &\left[\left(\mu_H(F)-\frac{\delta}{n-k}\right)\frac{n'-l}{n'}-\frac{c_1H}{n}\right]m+\left[\frac{\chi(E')}{n'}-\frac{\chi(E)}{n}\right]+\\
        &\left[\frac{1}{n-k}(c_1Hm+\chi(E)-n\chi(\mathcal{O}_X))-\frac{\epsilon}{n-k}\right]\left[\frac{k'}{n'}-\frac{k}{n}\right] \\ = 
        &\left[-\frac{\delta}{n-k}\frac{n'-l}{n'}+\frac{c_1H}{n'}\frac{k'-l}{n-k}\right]m+\left[\frac{\chi(E')}{n'}-\frac{\chi(E)}{n}\right]+\\&
        \left[\frac{1}{n-k}(\chi(E)-n\chi(\mathcal{O}_X))-\frac{\epsilon}{n-k}\right]
        \left[\frac{k'}{n'}-\frac{k}{n}\right].
    \end{array}}\]
    Since $\delta>0$ and  $k'-l \leq 0$, we conclude that $P^\alpha_{H,(E',V')}(m) - P^\alpha_{H,(E,V)}(m)<0$.

   \noindent \textbf{Case (C): $F'=0$.} In this case $E' \cong W$ and 
    \[ {\tiny \begin{array}{rl}
        P^\alpha_{H,(E',V')}(m) - P^\alpha_{H,(E,V)}(m)=(\frac{c_1'H}{n'}-
        \frac{c_1H}{n})m+\frac{\chi(E')}{n'}-\frac{\chi(E)}{n}+\left[\frac{1}{n-k}(c_1Hm+\chi(E)-n\chi(\mathcal{O}_X))-\frac{\epsilon}{n-k}\right]\left[\frac{k'}{n'}-\frac{k}{n}\right].\\
    \end{array}}\]
   Notice that the coefficient of $m$ in the above polynomial is strictly negative if and only if
   \[ (n-k)nc_1'H+n(k'-n')c_1H <0.\]
By Lemma \ref{rankdim}, $c_1'H \leq 0$ and $k' \leq n' $. If one of these inequalities is strict, then    $P^\alpha_{H,(E',V')}(m) - P^\alpha_{H,(E,V)}(m)<0$. Assume that $c_1'H=0$ and $k'= n' $.
In this case,
\begin{equation} \label{casC} P^\alpha_{H,(E',V')}(m) - P^\alpha_{H,(E,V)}(m)=\frac{\chi(E')}{n'}-\chi(\mathcal{O}_X)-\frac{\epsilon}{n-k}.\end{equation}
By Lemma \ref{rankdim}, $W \cong  \mathcal{O}_X^{\oplus k'}$ or $\chi(W)=k'\chi(\mathcal{O}_X)-c_2(W)$ with $c_2(W) > 0$. In the first case,  $(\ref{casC})$ is equivalent to
\[P^\alpha_{H,(E',V')}(m) - P^\alpha_{H,(E,V)}(m)=-\frac{\epsilon}{n-k} <0,\]
otherwise, 
\[P^\alpha_{H,(E',V')}(m) - P^\alpha_{H,(E,V)}(m)=-c_2(W)-\frac{\epsilon}{n-k} <0,\]
since $c_2(W) >0$. Hence, in both cases $P^\alpha_{H,(E',V')}(m) - P^\alpha_{H,(E,V)}(m)<0$. 
   
     \noindent \textbf{Case (D): $F'=F$ and $W \neq \mathcal{O}_X^{\oplus k'}$.} Denote by  $l=rk(W)$. In this case, 
      \[ {\tiny \begin{array}{rl}
        P^\alpha_{H,(E',V')}(m) - P^\alpha_{H,(E,V)}(m)  = & 
        \left[\frac{c_1(W)}{n-k+l}+\frac{c_1(F)}{n-k+l}-\frac{c_1}{n}\right]Hm+\frac{\chi(E')}{n-k+l}-\frac{\chi(E)}{n}+\\ & \left[\frac{1}{n-k}(c_1Hm+\chi(E)-n\chi(\mathcal{O}_X))-\frac{\epsilon}{n-k}\right]\left[\frac{k'}{n-k+l}-\frac{k}{n}\right]\\ 
        = & 
        \left[\frac{c_1(W)}{n-k+l}-c_1\left(\frac{1}{n-k+k}-\frac{1}{n}+\frac{1}{n-k}\left(\frac{k'}{n-k+l}-\frac{k}{n}\right)\right)\right]Hm+\frac{\chi(E')}{n-k+l}-\frac{\chi(E)}{n}+\\
        &\left[\frac{1}{n-k}(\chi(E)-n\chi(\mathcal{O}_X))-\frac{\epsilon}{n-k}\right]\left[\frac{k'}{n-k+l}-\frac{k}{n}\right]\\ = 
          & \frac{1}{n-k+l}\left[c_1(W)+c_1\frac{k'-l}{(n-k)}\right]Hm+\frac{\chi(E')}{n-k+l}-\frac{\chi(E)}{n}+ \\ &\left[\frac{1}{n-k}(\chi(E)-n\chi(\mathcal{O}_X))-\frac{\epsilon}{n-k}\right]\left[\frac{k'}{n-k+l}-\frac{k}{n}\right]
    \end{array}}\]

By Lemma \ref{rankdim}, $k'\leq l$ and $c_1(W)\cdot H \leq 0$. Thus, if one of these inequalities is strict, the coefficient of the above polynomial is strictly negative and hence   \[P^\alpha_{H,(E',V')}(m) - P^\alpha_{H,(E,V)}(m)<0.\] Assume that $k'=l$ and $c_1(W)\cdot H =0$. In this case, by Lemma   \ref{rankdim}, since $W \neq \mathcal{O}_X^{\oplus k'}$, $\chi(W)=k'\chi(\mathcal{O}_X)-c_2(W)$ with $c_2(W) > 0$. Hence, since $\chi(E')=\chi(E)-k\chi(\mathcal{O}_X)+\chi(W)$,
\[ P^\alpha_{H,(E',V')}(m) - P^\alpha_{H,(E,V)}(m)  = -\frac{c_2(W)}{n-k+l}-\frac{\epsilon(l-k)}{n(n-k+l)}.\]
Since $c_2 \geq 1$, this difference is negative for any $\epsilon$ such that $\epsilon(k-l)<n$. Therefore, $P^\alpha_{H,(E',V')}(m) - P^\alpha_{H,(E,V)}(m)<0$. 

    \noindent \textbf{Case (E): $F'=F$ and $W = \mathcal{O}_X^{\oplus k'}$.}  In this case, the quotient $E/E' \cong \mathcal{O}_X^{\oplus k}/\mathcal{O}_X^{\oplus k'}$ is of rank $k-k'$ and generated by $k-k'$ global sections. Thus $\mathcal{O}_X^{\oplus (k-k')}$ is a subbundle of $E$, and hence $E=E'\oplus \mathcal{O}_X^{\oplus (k-k')}$ which is impossible since $\{e_1, e_2, \ldots, e_k\}$ are linearly independent.
    
    Putting altogether we get that $(E,V)$ is $\alpha$-stable.  
\end{proof}

The following result determines an upper bound for the number of independent sections of $E$ in terms of their invariants and of the variety. This result is an analogue of Clifford’s Theorem for $\alpha-$semistable coherent systems on $X$ (Compare the first item with  \cite[Theorem 3.4]{CMR}).

\begin{Corollary}
    Let $(E,V)$ be an $\alpha$-semistable coherent system of type $(n;c_1,c_2,k)$ with $\alpha=\frac{1}{n-k}(c_1Hm+\chi(E)-n\chi(\mathcal{O}_X))-\frac{1}{n-k}\epsilon$ and $\epsilon>0$ sufficiently small.
    \begin{itemize}
        \item [(i)] Let H be a very ample divisor on $X$ such that $K_X\cdot H \leq 0$. Let $a$ be an integer such that
 \[H^2 \cdot\max\{\frac{n^2-1}{4},1\} < \frac{\binom{a+2}{2}-a-1}{a}\] and assume that $0\leq \frac{c_1H}{n}<aH^2+K_XH$. Then, \[h^0(E)\leq n+a\frac{c_1H}{2}. \]
 \item [(ii)] Let $X=\mathbb{P}^2$. Then, $h^0(E)\leq \frac{c_1^2}{2(n-k)}+\frac{3}{2}c_1+n.$
    \end{itemize}
\end{Corollary}

\begin{proof}
    By Proposition \ref{Fsemistable}, the coherent system $(E,V)$ fits into an extension $$0\rightarrow \mathcal{O}_X^{\oplus k}\rightarrow E\rightarrow F\rightarrow 0,$$
where $F$ is an $H$-semistable (Mumford-Takemoto) torsion free sheaf. From this exact sequence, 
\[h^0(E) \leq k + h^0(F).\]
The result $(i)$ follows that \cite[Theorem 4.1]{CM} and $(ii)$ follows by \cite[Theorem 1.1]{CHN}. 
\end{proof}

In Proposition \ref{E-stable} we have seen that any extension 
\begin{equation}
         0 \to \mathcal{O}_X^{\oplus k}  \to E \to F \to 0
     \end{equation}
     with $F$ an $H$-stable quotient defines an $\alpha$-stable coherent system for $\alpha$ big enough. Now the natural question to consider is what happens if $F$ is strictly $H$-semistable. Let us analyze this in the following result. 

\begin{Proposition}
     Any extension \begin{equation}\label{ext}
         0 \to \mathcal{O}_X^{\oplus k}  \to E \to F \to 0
     \end{equation}  in which $F$ is a rank $n-k>0$ strictly $H$-semitable bundle fails to define a coherent system $\alpha$-stable for any $$\alpha=\frac{1}{n-k}(c_1Hm+\chi(E)-n\chi(\mathcal{O}_X))-\frac{1}{n-k}\epsilon$$ with $\epsilon>0$ sufficiently small, if and only if there exists an extension 
     \begin{equation}\label{subdes}
         0 \to \mathcal{O}_X^{\oplus k'}  \to E' \to F' \to 0 
     \end{equation}
     and a commutative diagram 
     \begin{equation}\label{dia1}
        \begin{diagram}
        \node{0} \arrow{e}  \node{W} \arrow{e} \arrow{s} \node{E'} \arrow{e} \arrow{s} \node{F'} \arrow{e} \arrow{s} \node{0}\\
        \node{0} \arrow{e}  \node{\mathcal{O}_X^{\oplus k}} \arrow{e}  \node{E} \arrow{e}  \node{F} \arrow{e}  \node{0}
    \end{diagram}
    \end{equation}
    such that $0\neq F' \subset F$,  $\mu_H(F')=\mu_H(F)$, and 
    \begin{equation}\label{condi}
        \left(\frac{\chi(E')}{n'}-\frac{\chi(E)}{n}\right)+\left(\frac{1}{n-k} \left( \chi(E)-n\chi(\mathcal{O}_X)\right)-\frac{\epsilon}{n-k}\right)\left(\frac{k'}{n'}-\frac{k}{n}\right)>0.
    \end{equation}
\end{Proposition}

\begin{proof}
    The existence of \eqref{subdes} and the diagram \eqref{dia1} implies that the coherent system $(E,V)$ defined by \eqref{ext} is not $\alpha-$stable. 
     Conversely, if the coherent system $(E,V)$ defined by \eqref{ext} is not $\alpha-$stable, then there exists a diagram \eqref{dia1} for which \eqref{condi} holds. 
\end{proof}

\vspace{3mm}
\begin{Notation} \rm Fix natural numbers $n, k \in \mathbb{N}$ with $k<n$ and Chern classes $c_1$ and $c_2$. Let us denote by $G_L(n ;c_1,c_2,k)$  the moduli space of equivalence classes of coherent systems which are $\alpha$-stable for any $\alpha$  close to the bound  $(c_1Hm+\chi(E)-n\chi(\mathcal{O}_X))/(n-k)$ and denote by $G_L^s(n;c_1,c_2,k)$ the moduli space of equivalence classes of coherent systems $(E,V) \in G_L(n;c_1,c_2,k)$ such that the quotient is $H$-stable (see Proposition \ref{Fsemistable}). 
\end{Notation}

\begin{Theorem} \label{dimGL}
Let $n, k \in \mathbb{N}$ be  natural numbers with $k<n$, $X$  a smooth projective surface and $H$ an ample  divisor on $X$. Fix Chern classes $c_1$ and $c_2 \gg 0$ and  assume that $$((n-k)K_X+c_1)H \leq 0.$$ Then, the space $G_L^s(n;c_1,c_2,k)$ is isomorphic to a Grassmann bundle over $M_H(n-k;c_1,c_2)$. Moreover, it is irreducible with dimension
$$  \dim G_L^s(n;c_1,c_2,k)=(2n-k)c_2 +(1- n(n-k))\chi(\mathcal{O}_X)-k^2-k(\frac{c_1K}{2}+\frac{c_1^2}{2})-(n-k-1)c_1^2.$$
\end{Theorem}
\begin{proof}
    First of all notice that according to Proposition \ref{E-stable}, $G_L^s(n;c_1,c_2,k)$ can be naturally identified with  the family of extension classes of the form
$$ 0 \to \mathcal{O}_X^{\oplus k} \to E \to F \to 0,$$
where $F$ is an $H$-stable torsion free sheaf with $\rank(F)=n-k$ and  $c_i(F)=c_i$ for $i=1,2$. We will assume that  $n-k$ and $c_1H$ are coprime, in which case there exists a Poincar\'e sheaf 
$$ \mathcal{U} \to M_H(n-k;c_1,c_2) \times X $$
of $H$-stable bundles on $X$ such that for any $t \in M_H(n-k;c_1,c_2) $,
$$ \mathcal{U}|_{\{t\} \times X} \cong F_t$$
with $F_t$ a rank $n-k$, $H$-stable torsion free sheaf with $c_1(F)=c_1$ and $c_2(F)=c_2$ corresponding to $t \in M_H(n-k;c_1,c_2)$.
It is also possible to carry out the same construction without the assumption that $n-k$ and $c_1H$ are coprime, using
only the local existence of a universal sheaf on $M_H(n-k;c_1,c_2)$,  carrying out the constructions locally and showing the independence of the choice of the locally universal sheaf and conclude that the construction glue as a global algebraic object. 

Consider the natural projections
$$ \begin{array}{ccc}
    M_H(n-k;c_1,c_2) \times X  & \mapright{\pi} &  M_H(n-k;c_1,c_2) \\
  \quad  \downarrow{p} & & \\
X .\end{array}$$
Set $\mathcal{E}:=Ext^1_{\pi}(\mathcal{U}, p^*\mathcal{O}_X)$ where $Ext^1_{\pi}(\mathcal{U}, \cdot)$ is the right derived functor of
$$ \mathcal{H}{om}_{\pi}(\mathcal{U}, \cdot):=\pi_{*} \mathcal{H}{om}(\mathcal{U}, \cdot).$$
Notice that for any $t \in M_H(n-k;c_1,c_2)$, 
$$ \dim \Ext^1(\mathcal{U}|_{\{t\} \times X},\mathcal{O}_X)=\dim \Ext^1(F_t,\mathcal{O}_X)=h^1(F_t^*).$$
Since $F_t^*$ is $H$-stable and $c_1H \geq 0$, $h^0(F_t^*)=0$. On the other hand, if $0 \neq h^2(F_t^*)=h^0(F_t(K_X))$, by stability we would have  $((n-k)K_X+c_1)H >0$ which contradicts our assumption. Thus, $h^1(F_t^*)=-\chi(F_t^*)$ and the function $\dim \Ext^1(\mathcal{U}|_{\{t\} \times X},\mathcal{O}_X)$ is constant  on $t \in M_H(n-k;c_1,c_2)$. Therefore, the sheaf $\mathcal{E}$ is locally free of rank $p=-\chi(F_t^*)$  on $M_H(n-k;c_1,c_2)$ and compatible with base change. Finally, define the Grassmann bundle over $M_H(n-k;c_1,c_2)$
$$\mathcal{F}:= Gr(k,\mathcal{E}).$$
Notice that according to the above results, we have the diagramm
$$ \begin{array}{ccc}
    \mathcal{F}  & \mapright{g} &  M_H(n-k;c_1,c_2) \\
  \quad  \downarrow{f} & & \\
G_L^s(n;c_1,c_2,k) \end{array}$$
with fibers $g^{-1}(F)$ over $F \in M_H(n-k;c_1,c_2)$ identified with a non-empty open subset of $Grass(k,\Ext^1(F,\mathcal{O}_X))$ and by Proposition \ref{E-stable} and Proposition \ref{Fsemistable}, $\mathcal{F}$ is canonically identified with $G_L^s(n;c_1,c_2,k)$. Hence, 
$$ \begin{array}{r l}  
   \dim G_L^s(n;c_1,c_2,k) =  & \dim \mathcal{F}  \\
     = & kp-k^2+\dim M_H(n-k;c_1,c_2).
\end{array}$$
The assumption $c_2 \gg 0$, implies that the moduli space $M_H(n-k;c_1,c_2)$ is irreducible and has the expected dimension equal to 
$$ \dim M_H(n-k;c_1,c_2)=2(n-k)c_2-(n-k-1)c_1^2-((n-k)^2-1)\chi(\mathcal{O}_X).$$
By Riemman-Roch Theorem $p=-\chi(F_t^*)=- (n-k)\chi(\mathcal{O}_X)-\frac{c_1K}{2}-\frac{c_1^2}{2}+c_2$. Putting altogether we get $G_L^s(n;c_1,c_2,k)$ is irreducible with dimension 
$$(2n-k)c_2 +(1- n(n-k))\chi(\mathcal{O}_X)-k^2-k(\frac{c_1K_X}{2}+\frac{c_1^2}{2})-(n-k-1)c_1^2.$$
\end{proof}

\begin{Remark} \rm 
    Notice that in Theorem \ref{dimGL} we can omit the hypothesis $c_2\gg0$ if the moduli space $M_H(n-k;c_1,c_2)$ is irreducible, and of the expected dimension. In particular, for anticanonical rational surfaces we have that  the moduli space $M_H(n-k;c_1,c_2)$  has the expected dimension.  
\end{Remark}
In a more general setting, we can obtain a lower bound for the dimension of $G_L^s(n;c_1,c_2,k)$. 

\begin{Theorem} \label{dimGL2}
Let $n, k \in \mathbb{N}$ be  natural numbers with $k<n$, $X$  a smooth projective surface and $H$ an ample  divisor on $X$. Fix Chern classes $c_1$ and $c_2 \gg 0$.  Then, the space $G_L^s(n;c_1,c_2,k)$ is irreducible and isomorphic to a Grassmann bundle over an open set of  $M_H(n-k;c_1,c_2)$ and has dimension
$$  \dim G_L^s(n;c_1,c_2,k) \geq (2n-k)c_2 +(1- n(n-k))\chi(\mathcal{O}_X)-k^2-k(\frac{c_1K}{2}+\frac{c_1^2}{2})-(n-k-1)c_1^2.$$
\end{Theorem}
\begin{proof}
We will keep the notations introduced in the proof of Theorem \ref{dimGL}. In addition, denote by $\mathcal{H}:= \mathcal{H}{om}(\mathcal{U}, p^* \mathcal{O}_X)$. First of all, notice that the map
\[ \begin{array}{ccc} M_H(n-k;c_1,c_2) & \to & \mathbb{N} \\ 
t &\to & h^2(X \times \{ t\}, \mathcal{H}_t) \end{array} \]
is upper semicontinuous on $M_H(n-k;c_1,c_2)$. Let us denote by \[n_0=\min_{t \in M_H(n-k;c_1,c_2)}\{h^2( \mathcal{H}_t)\}.\] By semicontinuity, there exist a non empty open subset $M_H^0 \subset M_H(n-k;c_1,c_2) $ such that for any $t \in M_H^0$, $h^2( \mathcal{H}_t)$ is constant equal to $n_0$.  Moreover,  by stability $h^0( \mathcal{H}_t)=0$. Hence, over $M_H^0$, $\mathcal{E}=Ext^1_{\pi}(\mathcal{U}, p^*\mathcal{O}_X)$ is locally free of rank $p=-\chi(F_t^*)+n_0 \geq -\chi(F_t^*)$. Now, arguing step by step as in the proof of Theorem \ref{dimGL}, we get that $ G_L^s(n;c_1,c_2,k)$ is irreducible and 
$$  \dim G_L^s(n;c_1,c_2,k) \geq (2n-k)c_2 +(1- n(n-k))\chi(\mathcal{O}_X)-k^2-k(\frac{c_1K}{2}+\frac{c_1^2}{2})-(n-k-1)c_1^2.$$
\end{proof}

We will end the section with one example where the moduli space $G_L^s(n;c_1,c_2,k)$ is generically smooth. To this end, we need to recall the following definition
\begin{Definition}
    \label{steiner}A rank $r$ Steiner bundle on $\mathbb{P}^2$ is a vector bundle $F$ given by an exact sequence of the following type
    \[ 0 \longrightarrow \mathcal{O}_X(-1)^t \longrightarrow \mathcal{O}_X^{t+r} \longrightarrow F.  \]
    In particular, $c_1(F)=t$ and $c_2(F)= {t+1\choose 2}$. 
\end{Definition}

\vspace{3mm}
It is known that being Steiner is an open condition and that if $\rank(F)< \frac{c_1(F)(1 +\sqrt{5})}{2}$, then the set of Steiner bundles $F$ defines an open dense subset inside the corresponding moduli space $M_{\mathbb{P}^2}(r;c_1(F),c_2(F))$ (see \cite{Cos}). 

\begin{Proposition} \label{smooth}
Let $X=\mathbb{P}^2$ and consider $n,k,c_1 \in \mathbb{Z}_{>0}$ with $0<n-k < c_1 ( \frac{1+\sqrt{5}}{2})$. Then $G_L^s(n;c_1,{c_1+1 \choose 2},k)$ is generically smooth. 
\end{Proposition}
\begin{proof}
According to \cite{He}; Corollaire 3.14, it is enough to see that for a generic point $(E,V) \in G_L^s(n;c_1,{c_1+1 \choose 2},k)$, 
\[ \Ext^2((E,V),(E,V))=0.\]
Recall that under the assumption $n-k < c_1 ( \frac{1+\sqrt{5}}{2})$, the moduli space $M_{\mathbb{P}^2}(n-k,c_1,{c_1+1 \choose 2} )$ is irreducible and Steiner bundles $F \in M_{\mathbb{P}^2}(n-k,c_1,{c_1+1 \choose 2} )$ define an open dense subset. Thus, according to Proposition \ref{Fsemistable} and Proposition \ref{E-stable}, we can assume that a generic $(E,V) \in G_L^s(n;c_1,{c_1+1 \choose 2},k)$ defines an exact sequence 
\begin{equation}
     \label{ex4} 0\rightarrow \mathcal{O}_{\mathbb{P}^2}^{\oplus k}\rightarrow E\rightarrow F\rightarrow 0 \end{equation}
where $F$ is a stable Steiner bundle, that is, $F$ is given by the exact sequence 
\begin{equation} \label{ex5} 0\rightarrow \mathcal{O}_{\mathbb{P}^2}(-1)^{\oplus c_1}\rightarrow \mathcal{O}_{\mathbb{P}^2}^{\oplus n-k+c_1} \rightarrow F\rightarrow 0. \end{equation}
In particular, $H^1(F)=0$ and thus $H^1(E)=0$. 

On the other hand, by \cite{He}; Corollaire 1.6, we have the long exact sequence 
$$ \to \Hom(V, H^1(E)) \to \Ext^2((E,V),(E,V)) \to \Ext^2(E,E) \to .  $$
Since $H^1(E)=0$, we have $\Hom(V, H^1(E))=0$. Hence, it only remains to prove that  $\Ext^2(E,E)=0$ which easily follows from the exact sequences (\ref{ex4}) and (\ref{ex5}). 
\end{proof}

 \end{document}